\documentclass[11pt]{article}

\usepackage[margin=1in]{geometry}
\usepackage[T1]{fontenc}
\usepackage{amsmath,amssymb,amsthm}
\usepackage{xcolor}
\usepackage[most]{tcolorbox}
\usepackage[mono,sfdefault]{notomath}
\usepackage{titlesec}
\usepackage[hidelinks]{hyperref}

\titleformat*{\section}{\normalfont\rmfamily\Large\bfseries}
\titleformat*{\subsection}{\normalfont\rmfamily\large\bfseries}
\titleformat*{\subsubsection}{\normalfont\rmfamily\normalsize\bfseries}
\titleformat*{\paragraph}{\normalfont\rmfamily\normalsize\bfseries}
\titleformat*{\subparagraph}{\normalfont\rmfamily\normalsize\bfseries}

\definecolor{myred}{RGB}{215, 0, 21}
\definecolor{mygreen}{RGB}{36, 138, 61}
\definecolor{myblue}{RGB}{0, 64, 221}
\definecolor{myyellow}{RGB}{178, 80, 0}
\definecolor{mycyan}{RGB}{0, 113, 164}
\definecolor{myindigo}{RGB}{54, 52, 163}
\definecolor{myorange}{RGB}{201, 52, 0}
\definecolor{mypink}{RGB}{211, 15, 69}
\definecolor{mypurple}{RGB}{137, 68, 171}
\definecolor{myteal}{RGB}{0, 130, 153}
\definecolor{mygray}{RGB}{108, 108, 112}

\newcommand{\bfemph}[1]{\emph{\textbf{#1}}}

\newtheorem{theorem}{Theorem}[section]
\newtheorem{corollary}[theorem]{Corollary}

\theoremstyle{remark}

\theoremstyle{plain}

\tcolorboxenvironment{corollary}{%
  enhanced jigsaw,colframe=mycyan,interior hidden,
  before skip=10pt,after skip=10pt,parbox=false}

  \tcolorboxenvironment{lemma}{%
  enhanced jigsaw,colframe=mycyan,interior hidden,
  before skip=10pt,after skip=10pt,parbox=false}

\tcolorboxenvironment{theorem}{%
  enhanced jigsaw,colframe=mycyan,interior hidden,
  before skip=10pt,after skip=10pt,parbox=false}

\tcolorboxenvironment{proof}{%
  blanker,breakable,left=5mm,
  before skip=10pt,after skip=10pt,
  borderline west={1mm}{0pt}{myred},parbox=false}

\title{\rmfamily An elementary proof of a tight \(\Delta+1\) bound for the additive chromatic number}
\author{
    Alex Koutsoutis\\\texttt{akoutsoutis@tamu.edu}
    \and
    Joseph Koutsoutis\\\texttt{jsk245@scarletmail.rutgers.edu}
}
\date{}

\begin{document}

\maketitle

\begin{abstract}
\setlength{\listparindent}{0pt}
\noindent
An additive coloring assigns positive integer labels to the vertices of a graph so that adjacent vertices have different sums of their neighbors' labels.
The additive chromatic number \(\eta(G)\) is the least \(k\) for which labels from \([k]\) suffice.
We give an elementary orientation argument proving that every finite simple graph \(G\) satisfies \(\eta(G) \leq \Delta(G) + 1\).
This argument yields a \((\Delta(G) + 1)\)-additive-coloring in polynomial time and improves upon the current best bound of \(\eta(G) \leq \Delta(G)^{2} - \Delta(G) + 1\), valid for \(\Delta(G) \geq 2\), from Akbari, Ghanbari, Manaviyat, and Zare~\cite{AGMZ13}.
\end{abstract}

\section{Introduction}

All graphs in this paper are simple and finite.
Let \(G\) be a graph and \(f: V(G) \to \mathbb{N}^{*}\) a vertex coloring of \(G\).
For a vertex \(v \in V(G)\), write
\begin{equation*}
 S_{f}(v) = \sum_{w \in N_{G}(v)}f(w)\text{.}
\end{equation*}
If \(S_{f}(u) \neq S_{f}(v)\) for every edge \(\{u, v\} \in E(G)\), then \(f\) is an \bfemph{additive coloring} of \(G\).
The \bfemph{additive chromatic number} of \(G\), denoted by \(\eta(G)\), is the least positive integer \(k\) for which \(G\) admits an additive coloring with labels in \([k]\).

Additive colorings were introduced by Czerwi\'{n}ski, Grytczuk, and \.{Z}elazny~\cite{CGZ09} using the terms \bfemph{lucky labeling} and \bfemph{lucky number} for additive coloring and additive chromatic number, respectively.

For general graphs, the best known upper bound on \(\eta(G)\) in terms of the maximum degree has been quadratic.
Akbari, Ghanbari, Manaviyat, and Zare~\cite{AGMZ13} studied the list version of the problem: the \bfemph{lucky choice number} \(\eta_{\ell}(G)\) is the least \(k\) such that, whenever each vertex \(v\) is given a set \(L(v)\) of \(k\) positive integers, there is an additive coloring \(f\) with \(f(v) \in L(v)\) for every \(v\).
They proved that \(\eta_{\ell}(G) \leq \Delta(G)^{2} - \Delta(G) + 1\) for \(\Delta(G) \geq 2\).
Since \(\eta(G) \leq \eta_{\ell}(G)\), this gives
\begin{equation}\label{eq:upper-bound}
\eta(G) \leq \Delta(G)^{2} - \Delta(G) + 1, \qquad \Delta(G) \geq 2\text{.}
\end{equation}
Sever\'{\i}n~\cite{Sev20} describes this as the best known upper bound on \(\eta(G)\) in terms of \(\Delta(G)\), and we are not aware of any later improvement that holds for all graphs.

Gossett~\cite{Gos24} proved an additive analogue of the Alon--Tarsi theorem~\cite{AT92}: if some orientation of \(G\) satisfies an Alon--Tarsi-type condition, then lists whose sizes exceed the outdegrees by one suffice.
He conjectured that every graph has such an orientation; this conjecture would imply \(\eta_{\ell}(G) \leq \Delta(G) + 1\) for every graph \(G\).

We improve \eqref{eq:upper-bound} to \(\eta(G) \leq \Delta(G) + 1\), which is tight since \(\eta(K_{\Delta+1}) = \Delta + 1\).
In fact, we prove a stronger statement: for any choice of positive integers \(\{a_{v}\}_{v\in V(G)}\), there is an additive coloring \(f\) with \(f(v) \in \{a_{v}, a_{v} + 1, \ldots, a_{v} + d(v)\}\) for every vertex \(v\).
Moreover, this coloring is obtainable in polynomial time.

\section{Proof of the upper bound}

The proofs of the previous upper bounds rely on algebraic tools such as the Combinatorial Nullstellensatz and Alon--Tarsi-type orientation criteria~\cite{AGMZ13,CGZ09,Gos24}.
Our proof is elementary and constructive.

\begin{theorem}\label{thm:upper}
Let \(G\) be a graph, and assign a positive integer \(a_{v}\) to every vertex.
There is an additive coloring satisfying
\begin{equation*}
 f(v) \in \{a_{v}, a_{v} + 1, \ldots, a_{v} + d(v)\}, \qquad v \in V(G)\text{.}
\end{equation*}
In particular, \(\eta(G) \leq \Delta(G) + 1\).
\end{theorem}
\begin{proof}
For an orientation \(D\) of \(G\), put
\begin{equation*}
 f_{D}(v) = a_{v} + d_{D}^{+}(v), \qquad \Phi(D) = \sum_{\{x, y\} \in E(G)}f_{D}(x)f_{D}(y)\text{.}
\end{equation*}
Choose \(D\) minimizing \(\Phi\) among the finitely many orientations.
If \((u, v)\) is an arc in \(D\) and \(D'\) is the orientation with this one arc reversed, then we compute
\begin{align}
   f_{D'}(u)f_{D'}(v) - f_{D}(u)f_{D}(v) &= \left(a_{u} + d_{D'}^{+}(u)\right)\left(a_{v}+d_{D'}^{+}(v)\right) - \left(a_{u} + d_{D}^{+}(u)\right)\left(a_{v} + d_{D}^{+}(v)\right) \nonumber \\ 
   &= \left(a_{u} + d_{D}^{+}(u) - 1\right)\left(a_{v} + d_{D}^{+}(v) + 1\right) - \left(a_{u} + d_{D}^{+}(u)\right)\left(a_{v} + d_{D}^{+}(v)\right) \nonumber \\ 
   &= f_{D}(u) - f_{D}(v) - 1\text{.} \label{eq:1}
\end{align}
For \(w \neq u, v\), using similar computation to above, if \(\{u,w\} \in E(G)\),
\begin{equation}
   f_{D'}(u)f_{D'}(w) - f_{D}(u)f_{D}(w) = -f_{D}(w)\text{.} \label{eq:2}
\end{equation}
and if \(\{w, v\} \in E(G)\),
\begin{equation}
   f_{D'}(w)f_{D'}(v) - f_{D}(w)f_{D}(v) = f_{D}(w)\text{.} \label{eq:3}
\end{equation}
Combining \eqref{eq:1}, \eqref{eq:2}, and \eqref{eq:3} and writing \(f = f_{D}\) gives 
\begin{equation}\label{eq:4}
  \Phi(D') - \Phi(D) = S_{f}(v) - S_{f}(u) - 1\text{.}
\end{equation}
By minimality of \(D\), \(S_{f}(v) \geq S_{f}(u) + 1\) on every arc \((u, v)\). \qedhere
\end{proof}

This translates to a \((\Delta(G) + 1)\)-additive-coloring in polynomial time:

\begin{corollary}
After setting \(a_{v} = 1\) for all \(v \in V\) as in theorem~\ref{thm:upper}, start from any orientation and repeatedly reverse arcs whose endpoints have equal neighborhood sums.
This process terminates at an additive coloring after at most
\begin{equation*}
 |E(G)|\left((\Delta(G) + 1)^{2} - 1\right)
\end{equation*}
reversals.
\end{corollary}
\begin{proof}
Each reversal decreases \(\Phi\) by exactly \(1\) by \eqref{eq:4}.
Through the process, we maintain \(1 \leq f_{D}(v) \leq \Delta(G) + 1\).
Hence
\begin{equation*}
  |E(G)| \leq \Phi(D) \leq |E(G)|(\Delta(G) + 1)^{2}\text{.} \qedhere
\end{equation*}
\end{proof}

\section{Tool and computational resource disclosure}

The proofs in this paper were assisted by ChatGPT-6 Astra and Claude Opus 5.5.
We subsequently reviewed and refined the ideas, wrote them as shown above, and take full responsibility for the content of this paper.

\end{document}